\documentclass[11pt]{article}

\usepackage{color}
\usepackage{amsmath}
\usepackage{amsthm}
\usepackage{amsfonts}
\usepackage{amssymb}
\usepackage{multirow}
\usepackage[dvips]{graphicx}
\usepackage{epsfig}
\usepackage{latexsym}
\usepackage{enumerate}
\usepackage{tabularx}
\usepackage{booktabs}
\usepackage[caption=false]{subfig}
\usepackage[footnotesize]{caption}

\usepackage[bookmarksnumbered=false,]{hyperref}

\newtheorem{thm}{Theorem}[section]
\newtheorem{thm2}{Theorem}
\newtheorem{thm3}{Theorem}[section]
\newtheorem{thm4}{Theorem}[section]

\newtheorem{lem}[thm]{Lemma}
\newtheorem{prop}[thm]{Proposition}
\newtheorem{defn}[thm]{Definition}
\numberwithin{equation}{section}
\newtheorem{algorithm}[thm2]{Algorithm}
\newtheorem{rem}[thm3]{Remark}
\newtheorem{example}[thm4]{Example}

\def \spn {\,{\rm span}\,}
\def \argmin {\,{\rm argmin}\,}

\newcommand{\se}{\text{e}}
\newcommand{\si}{\text{i}}

\newcommand{\bR}{\mathbb{R}}

\newcommand{\bN}{\mathbb{N}}
\newcommand{\bZ}{\mathbb{Z}}

\newcommand{\va}{\boldsymbol{a}}

\newcommand{\vc}{\boldsymbol{c}}

\newcommand{\vhf}{\boldsymbol{\hat{f}}}

\newcommand{\cG}{\mathcal{G}}
\newcommand{\cH}{\mathcal{H}}

\newcommand{\mPsi}{\mathsf{\Psi}}
\newcommand{\mPhi}{\mathsf{\Phi}}
\newcommand{\mOmega}{\mathsf{\Omega}}

\newcommand{\tpsi}{\tilde{\psi}}

\begin{document}

\title{Approximating the Inverse Frame Operator from Localized Frames}

\author{Guohui Song\thanks{School of Mathematical and Statistical Sciences, Arizona State University, Tempe, AZ 85287. E-mail address: {\it gsong9@asu.edu}.}
\and Anne Gelb \thanks{School of Mathematical and Statistical Sciences, Arizona State University, Tempe, AZ 85287. E-mail address:{\it anne.gelb@asu.edu}. Supported in part by NSF-DMS-FRG award 0652833.}}

\date{}
\maketitle

\begin{abstract}
This investigation seeks to establish the practicality of numerical frame approximations.  Specifically, it develops a new method to approximate the inverse frame operator and analyzes its convergence properties.   It is established  that sampling with {\em well-localized frames} improves both the accuracy of the numerical frame approximation as well as the robustness and efficiency of the (finite) frame operator inversion. Moreover, in applications such as magnetic resonance imaging, where the given data often may not constitute a well-localized frame, a technique is devised to project the corresponding frame data onto a more suitable frame.  As a result, the target function may be approximated as a finite expansion with its asymptotic convergence solely dependent on its smoothness.   Numerical examples are provided.


\end{abstract}

\section{Introduction}\label{sec:introduction}
Due to their flexible nature, frames make useful representation tools for a variety of applications.  For example, in signal processing applications, the redundancy of frames is beneficial if signals are suspected of not capturing certain pieces of information.  Not enforcing orthogonality of traditional bases also is useful when small amounts of interference does not present too many difficulties, but working with a large (albeit orthogonal polynomial based) system does.  It is also possible that there are some functions that are better represented by frames than  by traditional orthogonal bases.    A nice introduction to frames in the context of some of these applications can be found in \cite{KC1, KC2}.

In several applications,  such as magnetic resonance imaging (MRI),
data may be collected as a series of non-uniform Fourier coefficients (see
e.g. \cite{BW00, pipe, sedarat, Aditya09}).  Since standard Fourier
reconstruction methods cannot be straightforwardly applied,  the
current methodology can generally be described as an interpolation or
approximation of the data onto Fourier integer coefficients which enables
image reconstruction via the Fast Fourier transform (FFT).\footnote{Most
often, of course, the target image is only piecewise smooth so
the Gibbs phenomenon is still evident in the reconstruction and must
be properly addressed.}  Convergence analysis for several common
MRI reconstruction algorithms was performed in \cite{Aditya09},
where it was shown that it is possible to post-process
the (interpolated) integer Fourier coefficients  to resolve the Gibbs 
phenomenon.   However, it was also demonstrated there that
the dominant reconstruction error was due to ``resampling''
the non-integer data onto integer coefficients, typically at best
${\mathcal O}(1/N^2)$ for given $m = {\mathcal O}(N)$ coefficients.  
Since then, in \cite{GelbHines} it was suggested that in such applications
it might be better {\em not to} resample the non-integer coefficients,
and thereby avoid the resampling error entirely.  In fact, even for
piecewise smooth functions, if the original data set constitutes a finite
number of Fourier frame coefficients, then the 
Gibbs phenomenon can be removed directly by using the same
post-processing techniques as in the uniform case.  In particular,
in \cite{GelbHines}, the spectral reprojection method, \cite{GS97},
was shown to yield {\em exponential convergence} in this case.  It was further
shown there that even if the original data could not be considered
as a finite set of coefficients of the truncated Fourier frame expansion
(i.e., the corresponding infinite sequence did not form a Fourier frame), 
the same reconstruction methods could still be applied, although not with
exponential accuracy.  

One of the main difficulties in approximating a function from its
frame coefficients, independent of its smoothness properties,
lies in the construction of the (finite) inverse frame
operator.  The frame algorithm devised in \cite{DS52} and 
accelerated in \cite{MR1955936, Grochenig1993} is iterative
and its speed greatly depends on the frame bounds.  Other iterative
methods can also be used, but inherently depend on what is known about
the frame bounds.  Furthermore, conditions that guarantee the overall
convergence of a truncated frame expansion are not well understood.
Hence the usefulness of numerical frame approximations is not
yet well established.

In this investigation we seek to establish the practicality of numerical frame approximations by developing a new approximation method for the inverse frame operator.  We establish  that sampling with {\em well-localized frames} improves both the accuracy of the numerical frame approximation as well as the robustness and efficiency of the (finite) frame operator inversion. Moreover, in applications such as magnetic resonance imaging, where the given data often may not constitute a well-localized frame, a technique is devised to project the corresponding frame data onto a more suitable frame.  As a result, the target function may be approximated as a finite expansion with its asymptotic convergence solely dependent on its smoothness.  
If the target function is only piecewise smooth, it is possible to apply high order post-processing methods, as demonstrated in \cite{GelbHines},  to remove the Gibbs phenomenon.

The paper is organized as follows:
Section \ref{sec:review} reviews some fundamental aspects of frame theory.  In Section \ref{section:localized} we establish the convergence rate of the Casazza-Christensen method of approximating the inverse frame operator for well-localized frames, \cite{Casazza2000, Christensen2000}. However, the convergence rate fails to hold when the sampling frame is not well-localized. To overcome this difficulty we propose a new method of approximating the inverse frame operator and prove its convergence rate in Section \ref{section:non-localized}.  In Section \ref{section:reconstruction} we use this approximation technique to develop a new numerical frame approximation method.  We demonstrate the effectiveness of our method  with some numerical experiments.  Concluding remarks are provided in Section \ref{section:conclusion}.

\section{Sampling with Frames}
\label{sec:review}
Let us first review the definition of frame (see  \cite{Christensen2003} for more details).  
\begin{defn}\label{def_frame}
Let $\cH$ be a separable Hilbert space and let $\{\psi_j\}_{j=1}^\infty$ be a frame for $\cH$ with bounds $A$ and $B$. That is, we have for all $f\in \cH$
\begin{equation}\label{framebound}
A\| f\|^2 \leq \sum_{j=1}^\infty \left| \langle f, \psi_j \rangle\right|^2 \leq B \| f\|^2, \quad A,B>0.
\end{equation}
The {\it frame operator} $S: \cH \rightarrow \cH$ is defined as
\begin{equation*}
Sf:= \sum_{j=1}^\infty \langle f, \psi_j \rangle \psi_j, \quad f\in \cH.
\end{equation*}
\end{defn}
Note that the frame operator $S$ is bounded invertible by the frame condition,  \eqref{framebound}. Moreover, any function $f\in \cH$ can be recovered from the sampling data $\{\langle f, \psi_j \rangle\}_{j=1}^\infty$ by
\begin{equation}
\label{framesum}
f=\sum_{j=1}^\infty \langle f, \psi_j \rangle \tpsi_j,
\end{equation}
where 
\begin{equation}\label{dualframe}
\tpsi_j:=S^{-1}\psi_j, \,j\in \bN 
\end{equation}
is called the {\it dual frame}.

Since $S^{-1}$ is generally not available in closed form, it will be necessary to construct $\tilde{S}_N^{-1}$, a finite-dimensional subspace approximation corresponding to ${\mathcal O}(N)$ finitely sampled frame coefficients or an ${\mathcal O}(N)$ truncated series expansion. A general method of approximating the inverse frame operator $S^{-1}$ was proposed  in \cite{Casazza2000} and its convergence was discussed in \cite{Casazza2000, Christensen2000} (see also \cite{Christensen2003}). In what follows, we will call this technique the Casazza-Christensen method. Note that the convergence rate for this method has yet to be established.

Our investigation seeks to establish the convergence rate of inverse frame operators under a certain set of constraints, which is essential in developing numerical frame approximations.  To this end, we will use the concept of {\em localized frames} \cite{Grochenig2004}:
\begin{defn}\label{def_localizedframe}
Let $\{\psi_j\}_{j=1}^\infty$ be a frame as defined in Definition \ref{def_frame}.  We say that $\{\psi_j\}_{j=1}^\infty$ is localized with respect to the Riesz basis $\{\phi_j\}_{j=1}^\infty$ with decay $s>0$ if
\begin{equation}\label{localizationassumption}
\left|\langle \psi_j, \phi_l \rangle\right| \leq c (1+ |j-l|)^{-s} \quad \mbox{and } \left|\langle \psi_j, \tilde{\phi}_l \rangle\right| \leq c (1+ |j-l|)^{-s}, \quad c>0, j,l\in \bN.
\end{equation}  
\end{defn}
The convergence rate of the numerical approximation to the inverse frame operator is directly related to the localization factor $s$.  For example, when $\{\phi_j\}_{j=1}^\infty$ is an orthonormal basis, it was shown in \cite{Christensen2005, Grochenig2010} that the {\em finite section method} approximates the inverse frame operator with a convergence rate dependent on localization rate $s$. The finite section method first establishes an bi-infinite linear system with the coefficients in the frame expansion and then approximates the solution by truncating the system:  
\begin{algorithm}\label{finitesection}(Finite Section Method \cite{Christensen2005, Grochenig2010})
Suppose $\{\psi_j\}_{j = 1}^\infty$ is a frame and we wish to approximate its inverse frame operator $S^{-1}$. That is, for a given function $f$, we wish to approximate $g=S^{-1}f$. Suppose further that $\{\phi_i\}_{i = 1}^\infty$ is an orthonormal basis.  
\begin{enumerate}
\item Define $g = \sum_{i = 1}^\infty \hat{g}_i \phi_i$ where $\hat{g}_i  = \langle g, \phi_i\rangle$ are the basis coefficients.
\item To determine $g$, it is equivalent to find $\{\hat{g}_i\}$. We consider $Sg = f$ for $g$ defined above.
\item Taking the inner products of both sides with the orthonormal basis $\{\phi_j\}_{j=1}^\infty$, we have 
$$\langle Sg, \phi_j \rangle = \langle f,\phi_j \rangle  = \hat{f}_j,\hspace{.2in}j = 1,2,...$$
\item The definition of $g$ then yields $\sum_{i = 1}^\infty \hat{g}_i \langle S\phi_i, \phi_j\rangle = \hat{f}_j$, where
\mbox{$S\phi_i = \sum_{k= 1}^\infty \langle \phi_i, \psi_k \rangle \psi_k$} by Definition \ref{def_frame}.
\item We solve the system for $\hat{g}_i$.  The maximum truncation values to ensure numerical stability and accuracy for the system are discussed in Remark \ref{rem:remark1}. 
\end{enumerate}
\end{algorithm}

In \cite{Christensen2005, Grochenig2010} it was shown that the convergence rate for Algorithm \ref{finitesection} is $s-1$ for the localization factor $s$ given in (\ref{localizationassumption}).  However, it is important to note that the method is applicable only for frames localized to Riesz bases, and is {\em not} directly applicable to the more general case of {\it intrinsically (self) localized frames}:
\begin{defn}\label{self-localization}
Let $\psi$  be a frame as defined in Definition \ref{def_frame}.  We say that $\psi$ is {\em intrinsically (self) localized} if
\begin{equation}\label{self-localization}
\left|\langle \psi_j, \psi_l \rangle\right| \leq c_0 (1+ |j-l|)^{-s}, \quad c_0>0, j,l\in \bN,
\end{equation}
with $s>1$.  
\end{defn}
Note that localization with respect to a Riesz basis \eqref{localizationassumption} implies the self-localization \eqref{self-localization} \cite{Fornasier2005}. When the sampling frame $\{\psi_j\}_{j=1}^\infty$ is intrinsically localized, we focus on the convergence rate of the Casazza-Christensen method proposed  in \cite{Casazza2000}. In fact, in \S \ref{section:localized} we show that this method yields better convergence than the finite section method given even less sampling data. However, as we also demonstrate in \S \ref{section:localized}, the convergence rate for the Casazza-Christensen method is $s-1/2$, very slow for small $s$ and this rate is failing to hold for $s\leq 1$. This result can be quite restrictive in several applications.  For example, new data collection techniques in magnetic resonance imaging (MRI) acquires a (finite) sampling of Fourier data on a spiral trajectory, \cite{pipe, sedarat}.  The corresponding Fourier frame \cite{BW00} is not well localized. Motivated by a desire to improve the quality of images reconstructed from sampling with Fourier frames (also called non-uniform Fourier data in the medical imaging literature), in Section \ref{section:non-localized} we also investigate the approximation of the inverse frame operator when the sampling frame $\{\psi_j\}_{j=1}^\infty$ is {\it not well-localized}, that is, when \eqref{self-localization} holds for some $s\in (\frac{1}{2}, 1]$. Specifically, to improve the convergence behavior, we introduce a new frame with an admissible localization rate and make use of the projection onto the finite-dimensional subspace by a well-localized frame (rather than using the original sampling frame with the slow localization rate). We remark that while our method is useful for sampling frames with $s \le 1$, it can also be used to improve the convergence rate for sampling frames with $s > 1$, but greater localization may be desired for the particular application.

\section{Constructing $S^{-1}$ for Localized Frames}\label{section:localized}
Let us assume that the sampling frame $\{\psi_j\}_{j=1}^\infty$ is intrinsically localized, that is, it satisfies \eqref{self-localization} for some $s>1$. In this section we analyze the convergence properties and establish the rate of convergence for the Casazza-Christensen method for approximating the inverse frame operator under this localization assumption. The method is reviewed below. (For more details, see \cite{Casazza2000, Christensen2003}.) 

For any $n\in \bN$, we let $\cH_n:=\spn\{\psi_j: 1\leq j\leq n\}$ be the finite-dimensional subspace of $\cH$ for $n\in \bN$. The finite subset $\{\psi_j\}_{j=1}^n$ is a frame for $\cH_n$ (c.f. \cite{Christensen2003}) with the frame operator $S_n: \cH \rightarrow \cH_n$ given by
\begin{equation}\label{finiteframe}
S_n f:= \sum_{j=1}^n \langle f, \psi_j \rangle \psi_j, \quad f\in \cH.
\end{equation}
Let $P_n$ be the projection from $\cH$ onto $\cH_n$ and let $V_n$ be the restriction of $P_nS_m$ on $\cH_n$. That is,
\begin{equation*}
V_n:=P_nS_m\mid_{\cH_n}.
\end{equation*}
It was shown in \cite{Casazza2000} that for any $n\in \bN$, there always  exists a large enough $m = m(n) \in \bN$ depending on $n$ such that $V_n$ is invertible in $\cH_n$ and for all $f\in \cH$
\begin{equation}\label{generalmethod}
V_n^{-1}P_n f \rightarrow S^{-1}f, \quad \mbox{as } n\rightarrow \infty.  
\end{equation}
However, the {\em rate} of convergence for (\ref{generalmethod}), which we will in sequel refer to as the Casazza-Christensen method, was not discussed.

To establish the convergence properties of (\ref{generalmethod})  for intrinsically localized frames, \eqref{self-localization}, we observe that
\begin{equation}
\label{errorequation}
\| S^{-1}f - V_n^{-1}P_n f\| \leq  \|S^{-1}f  - P_n S^{-1} f\| + \|P_n S^{-1} f - V_n^{-1}P_nS_m S^{-1} f \|+ \| V_n^{-1}P_nS_m S^{-1} f - V_n^{-1}P_n f\|.
\end{equation}
Since  $V_n^{-1}P_nS_m g=g$ for $g\in \cH_n$,  we have
\begin{equation*}
\|P_n S^{-1} f - V_n^{-1}P_nS_m S^{-1} f \| = \|V_n^{-1}P_nS_m(P_n S^{-1}f -S^{-1}f)  \|
\end{equation*}
for the second term on the right hand side of (\ref{errorequation}).  Also,  $P_n S^{-1} f \in \cH_n$, $V_n^{-1}P_nS_m P_n S^{-1} f =P_n S^{-1} f$ implies that the third terms on the right hand side of (\ref{errorequation}) can be rewritten as
\begin{equation*}
\| V_n^{-1}P_nS_m S^{-1} f - V_n^{-1}P_n f\| = \|V_n^{-1}P_n (S_m - S)S^{-1}f \|.
\end{equation*}
It therefore follows that
\begin{eqnarray}\label{error-decomposition-inv}
\nonumber &&\| S^{-1}f - V_n^{-1}P_n f\|\\
& \leq&  \|S^{-1}f  - P_n S^{-1} f\| + \|V_n^{-1}P_nS_m(S^{-1}f -P_n S^{-1}f)  \|+  \|V_n^{-1}P_n (S- S_m)S^{-1}f \|.
\end{eqnarray}
We will estimate the three error terms on the right hand side of the above inequality separately.  As it turns out, the self-localization property,
(\ref{self-localization}), is fundamental in analyzing convergence.

We begin with an estimate of the first term $\|S^{-1}f - P_n S^{-1} f\| $. To this end, we impose the following assumption on the decay of frame coefficients $\langle f, \psi_j\rangle$:
\begin{equation}\label{framecoeff}
\left|\langle f, \psi_j\rangle\right| \leq c j^{-s}, \quad c>0, j\in \bN.
\end{equation}
We also recall that the dual frame, (\ref{dualframe}) has the same localization property as the original frame, \cite{Fornasier2005}.  Hence \eqref{self-localization} implies that there exists a positive constant $c$ such that
\begin{equation}\label{duallocalization}
\left|\langle \tpsi_j, \tpsi_l \rangle\right| \leq c (1+ |j-l|)^{-s}, \quad  j,l\in \bN.
\end{equation}

To estimate $\| S^{-1}f - P_n S^{-1} f\|$, we will first need the results of two lemmas.
The first lemma is a result for the decay of the inner product of $f$ with the dual frame $\{\tpsi_j\}_{j=1}^\infty$ of $\{\psi_j\}_{j=1}^\infty$: 
\begin{lem}\label{lem-dualcoeff}
If assumptions \eqref{self-localization} and \eqref{framecoeff} hold, then there exists a positive constant $c$ such that
\begin{equation*}
\left|\langle f, \tpsi_j\rangle\right| \leq c j^{-s}, \quad j\in \bN,
\end{equation*}
where $\tpsi_j$ is given in (\ref{dualframe}).
\end{lem}
\begin{proof}
It follows from the dual frame property, \cite{Christensen2003}, that $f=\sum_{l=1}^\infty \langle f, \psi_l\rangle \tpsi_l$. By direct calculation, we have for any $j\in \bN$ that
\begin{equation*}
\left|\langle f, \tpsi_j\rangle\right|= \left| \sum_{l=1}^\infty \langle f, \psi_l\rangle \langle \tpsi_l, \tpsi_j\rangle\right|.
\end{equation*}
By \eqref{self-localization} and \eqref{framecoeff}, there exists a positive constant $c$ such that for all $j\in \bN$
\begin{equation*}
\left|\langle f, \tpsi_j\rangle\right|\leq c \sum_{l=1}^\infty l^{-s}(1+|j-l|)^{-s}.
\end{equation*}
The summation in the above inequality is bounded by a multiple of $j^{-s}$ according to a lemma in \cite{Grochenig2004},
which finishes the proof.
\end{proof}

The second lemma estimates $\|f - P_n f \|$ when $f$ satisfies \eqref{framecoeff} and $P_n$ is the projection from $H$ to $\cH_n$.
\begin{lem}\label{lem-Pn}
If assumptions \eqref{self-localization} and \eqref{framecoeff} hold, then there exists a positive constant $c$ such that
\begin{equation*}
\|f-P_n f \| \leq c n^{-(s-1/2)}.
\end{equation*}
\end{lem}
\begin{proof}
We define $T_nf:=\sum_{j=1}^n \langle f, \tpsi_j\rangle \psi_j$. Clearly $T_nf\in \cH_n$. Since $P_n$ is the projection from $H$ to $\cH_n$, we have 
\begin{equation*}
\|f-P_n f \|  \leq \|f-T_n f \|. 
\end{equation*}
It suffices to show $\|f-T_n f \|\leq c n^{-(s-1/2)}$ for some positive constant $c$, 
which we show by direct calculation. Since $f=\sum_{l=1}^\infty \langle f, \tpsi_l\rangle \psi_l$, we have that
\begin{equation*}
\|f-T_n f \|^2 = \left\| \sum_{l=n+1}^\infty \langle f, \tpsi_l\rangle \psi_l \right\|^2 = \sum_{j,l=n+1}^\infty \langle f, \tpsi_j\rangle \langle f, \tpsi_l\rangle \langle \psi_j, \psi_l\rangle.
\end{equation*}
It follows from Lemma \ref{lem-dualcoeff} and \eqref{self-localization} that there exists a positive constant $c$ such that
\begin{equation*}
\|f-T_n f \|^2  \leq c \sum_{j,l=n+1}^\infty j^{-s}l^{-s}(1+|j-l|)^{-s}.
\end{equation*}
Note that $j$ and $l$ are symmetric in the above inequality, and therefore
\begin{equation*}
\|f-T_n f \|^2  \leq 2c \sum_{j=n+1}^\infty \left[ j^{-s} \sum_{l=j}^\infty l^{-s}(1+|j-l|)^{-s}\right].
\end{equation*}
Since $\sum_{l=j}^\infty l^{-s}(1+|j-l|)^{-s} \leq j^{-s} \sum_{l=j}^\infty (1+l-j)^{-s} \leq \frac{1}{s-1}j^{-s} $, we have
\begin{equation*}
\|f-T_n f \|^2  \leq \frac{2c}{s-1} \sum_{j=n+1}^\infty j^{-2s} \leq \frac{2c}{s-1} \frac{1}{2s-1}n^{-(2s-1)},
\end{equation*}
yielding the desired result.
\end{proof}

We are now ready to estimate the first term $\|S^{-1}f - P_n S^{-1} f\| $ in \eqref{error-decomposition-inv}:
\begin{prop}\label{prop-Pn}
Assume \eqref{self-localization} and \eqref{framecoeff} hold. Then there exists a positive constant $c$ such that
\begin{equation}
\label{prop_Pn_eq}
\|S^{-1}f - P_n S^{-1} f \| \leq c n^{-(s-1/2)}.
\end{equation}
\end{prop}
\begin{proof}
By Lemma \ref{lem-dualcoeff}, there exists a positive constant $c$ such that $\left|\langle f, \tpsi_j\rangle\right| \leq c j^{-s}$ for all $j\in \bN$. Since $\langle S^{-1}f, \psi_j\rangle = \langle f, S^{-1}\psi_j\rangle = \langle f, \tpsi_j\rangle$, the function $S^{-1}f$ also satisfies \eqref{framecoeff}. This combined with Lemma \ref{lem-Pn} implies the desired result.
\end{proof}

We next estimate the second error term $\|V_n^{-1}P_nS_m(S^{-1}f -P_n S^{-1}f)  \|$ in \eqref{error-decomposition-inv}. Using (\ref{prop_Pn_eq}) and the fact that  $\|P_n\|\leq 1$ and $\| S_m\|\leq \|S\| \leq B$, we see that we only must estimate $\| V_n^{-1}\|$. In fact, we shall show that we can choose $m = m(n)$ such that $\| V_n^{-1}\|$ is uniformly bounded for all $n\in \bN$. To this end, we introduce the following constant 
\begin{equation}\label{Amn}
A_{m,n}:=\frac{c_0^2}{(2s-1)\lambda_{\min}(\mPsi_n)} n (m-n)^{-(2s-1)},
\end{equation}
where $\mPsi_n:=\left[\langle \psi_j, \psi_l\rangle \right]_{j,l=1}^n$ and $\lambda_{\min}(\mPsi_n)$ is its smallest eigenvalue. We assume here that the matrix $\mPsi_n$ is invertible. Otherwise, we can use its invertible principle sub-matrix instead and the same analysis can be carried over. Note that $m$ is always chosen to be greater than $n$. We first bound $\| V_n^{-1}\|$ for $m$.

\begin{lem}\label{lem-Vn}
Suppose \eqref{self-localization} holds. If $A_{m,n}<A$, then $\{P_n\psi_j\}_{j=1}^m$ is a frame for $\cH_n$ with frame bounds $A-A_{m,n}, B$ and frame operator $V_n$. Moreover,
\begin{equation*}
\|V_n^{-1}\| \leq \frac{1}{A-A_{m,n}}.
\end{equation*}
\end{lem}
\begin{proof}
We proceed by establishing the frame condition, \eqref{framebound}, by direct calculation. For any $g\in \cH_n$, we have $P_ng=g$, which implies that
\begin{equation*}
\sum_{j=1}^m |\langle g, P_n \psi_j \rangle|^2 = \sum_{j=1}^m |\langle P_n g,  \psi_j \rangle|^2 = \sum_{j=1}^m |\langle g, \psi_j \rangle|^2.
\end{equation*}
To see the upper bound, we observe that
\begin{equation*}
\sum_{j=1}^m |\langle g, \psi_j \rangle|^2 \leq \sum_{j=1}^\infty |\langle g, \psi_j \rangle|^2  \leq B \| g\|^2.
\end{equation*}
To estimate the lower bound, we first approximate $\sum_{j=m+1}^\infty |\langle g, P_n \psi_j \rangle|^2$. Since $g\in \cH_n$, there exists some $\va\in \bR^n$ such that $g=\sum_{j=1}^n a_j\psi_j$. It follows that
\begin{equation*}
\sum_{j=m+1}^\infty |\langle g, \psi_j \rangle|^2 = \sum_{j=m+1}^\infty \left|\sum_{l=1}^n a_l \langle \psi_l, \psi_j \rangle \right|^2 \leq   \| \va\|^2\sum_{j=m+1}^\infty  \left|\sum_{l=1}^n \langle \psi_l, \psi_j \rangle \right|^2.
\end{equation*}
Applying \eqref{self-localization} to the last term in the above inequality yields 
\begin{equation*}
\sum_{j=m+1}^\infty |\langle g, \psi_j \rangle|^2 \leq c_0^2 \| \va\|^2\sum_{j=m+1}^\infty  \sum_{l=1}^n (1+|j-l|)^{-2s}.
\end{equation*}
Since $m>n$, when $j\geq m+1$,  $\sum_{l=1}^n (1+|j-l|)^{-2s} \leq n (1+ j-n)^{-2s}$. It follows that
\begin{equation*}
\sum_{j=m+1}^\infty |\langle g, \psi_j \rangle|^2 \leq c_0^2 \| \va\|^2\sum_{j=m+1}^\infty n (1+j-n)^{-2s} \leq \frac{c_0^2 \| \va\|^2}{2s-1} n(m-n)^{2s-1}.
\end{equation*}
Note that $\| g\|^2 = \|\sum_{j=1}^n a_j\psi_j \|^2 = \va^T \mPsi_n \va \geq \lambda_{\min}(\mPsi_n) \|\va\|^2$. Substituting this into the above inequality yields 
\begin{equation*}
\sum_{j=m+1}^\infty |\langle g, \psi_j \rangle|^2 \leq A_{m,n} \|g \|^2.
\end{equation*}
Consequently, 
\begin{equation*}
\sum_{j=1}^m |\langle g, \psi_j \rangle|^2 = \sum_{j=1}^\infty |\langle g, \psi_j \rangle|^2 - \sum_{j=m+1}^\infty |\langle g, \psi_j \rangle|^2 \geq (A-A_{m,n}) \|g \|^2,
\end{equation*}
i.e., $A-A_{m,n}$ is the lower frame bound of the frame $\{P_n\psi_j\}_{j=1}^m$ for $\cH_n$ if $A_{m,n} <A$.

To show $V_n$ is the associated frame operator, we observe that for $g\in \cH_n$
\begin{equation*}
\sum_{j=1}^m \langle g, P_n \psi_j \rangle P_n\psi_j = \sum_{j=1}^m \langle P_n g, \psi_j \rangle P_n\psi_j = \sum_{j=1}^m \langle g, \psi_j \rangle P_n\psi_j= P_nS_m g = V_n g.
\end{equation*}
The bound of $\|V_n^{-1}\| $ follows immediately.
\end{proof}

We next give an estimate of $\|V_n^{-1}P_nS_m(S^{-1}f -P_n S^{-1}f)  \|$ by choosing $m$ such that $\|V_n^{-1}\| $ is uniformly bounded for all $n\in \bN$.

\begin{prop}\label{prop-2nderror}
Suppose assumptions \eqref{self-localization} and \eqref{framecoeff} hold. If we let
\begin{equation}\label{choose-m}
m = n + \left[\frac{2 n}{A (2s-1) \lambda_{\min}(\mPsi_n)}\right]^{\frac{1}{2s-1}},
\end{equation}
then $\|V_n^{-1}\|\leq 2/A$ and there exists a positive constant $c$ such that
\begin{equation*}
\|V_n^{-1}P_nS_m(S^{-1}f -P_n S^{-1}f)  \| \leq c n^{-(s-1/2)}.
\end{equation*}
\end{prop}
\begin{proof}
The bound of $\|V_n^{-1}\|$ follows from substituting \eqref{choose-m} into \eqref{Amn} and applying Lemma \ref{lem-Vn}. Moreover,
\begin{equation*}
\|V_n^{-1}P_nS_m(S^{-1}f -P_n S^{-1}f)  \| \leq \|V_n^{-1}\| \|P_n\| \| S_m\| \|S^{-1}f -P_n S^{-1}f \| \leq \frac{2B}{A} \|S^{-1}f -P_n S^{-1}f \| .
\end{equation*} 
This combined with Proposition \ref{prop-Pn} yields the desired result.
\end{proof}

It remains to estimate the last term $\|V_n^{-1}P_n (S- S_m)S^{-1}f \|$ in \eqref{error-decomposition-inv}. We have the following result.
\begin{prop}\label{prop-3rderror}
Suppose assumptions \eqref{self-localization} and \eqref{framecoeff} hold. If we choose $m$ as in \eqref{choose-m}, then there exists a positive constant $c$ such that
\begin{equation*}
\|V_n^{-1}P_n (S- S_m)S^{-1}f \|\leq c n^{-(s-1/2)}.
\end{equation*}
\end{prop}
\begin{proof}
By Proposition \ref{prop-2nderror}, $\|V_n^{-1}\|\leq 2/A$. Since $\|P_n\|\leq 1$, it suffices to show that  
\begin{equation}\label{prop-3rderror-1}
\|(S- S_m)S^{-1}f \|\leq c n^{-(s-1/2)}
\end{equation}
for some positive constant $c$. It follows from direct calculation that
\begin{equation*}
\|(S- S_m)S^{-1}f \|^2 = \sum_{j=m+1}^\infty \sum_{l=m+1}^\infty \langle S^{-1}f, \psi_j\rangle \langle S^{-1}f, \psi_l\rangle \langle \psi_j, \psi_l\rangle = \sum_{j=m+1}^\infty \sum_{l=m+1}^\infty \langle f, \tpsi_j\rangle \langle f, \tpsi_l\rangle \langle \psi_j, \psi_l\rangle.
\end{equation*} 
By assumptions \eqref{self-localization} and \eqref{framecoeff}, there exists a positive constant $c$ such that
\begin{equation*}
\|(S- S_m)S^{-1}f \|^2 \leq c \sum_{j=m+1}^\infty \sum_{l=m+1}^\infty j^{-s}l^{-s}(1+|j-l|)^{-s}.
\end{equation*} 
Note that we already show in Lemma \ref{lem-Pn} that the above summation term is bounded by $\frac{1}{(s-1)(2s-1)}n^{-(2s-1)}$, which implies \eqref{prop-3rderror-1}.
\end{proof}

We now summarize estimates of the three error terms in \eqref{error-decomposition-inv} to obtain an estimate for $\| S^{-1}f - V_n^{-1}P_n f\|$.
\begin{thm}\label{thm-inv}
Suppose assumptions \eqref{self-localization} and \eqref{framecoeff} hold. If we choose $m$ as in \eqref{choose-m}, then there exists a positive constant $c$ such that
\begin{equation*}
\| S^{-1}f - V_n^{-1}P_n f\| \leq c n^{-(s-1/2)}.
\end{equation*}
\end{thm}
\begin{proof}
It follows immediately from substituting estimates in Propositions \ref{prop-Pn}, \ref{prop-2nderror}, and \ref{prop-3rderror} into the error decomposition \eqref{error-decomposition-inv}.
\end{proof}

We close this section with two remarks:
\begin{rem}\label{rem:remark1}
In \cite{Christensen2005}, the sampling frame is assumed to be localized with respect to an orthonormal basis  and $m$ is chosen to be ${\mathcal O}(n^{-\frac{s}{s-1}})$ to obtain the optimal convergence rate $n^{-(s-1)}$.
\end{rem}

\begin{rem}\label{rem:remark2}
When $\{\psi_j\}_{j=1}^\infty$ is a Riesz basis, $\lambda_{\min}(\mPsi_n)$ is uniformly bounded below for all $n\in \bN$. We see that $m$ in \eqref{choose-m} is ${\mathcal O}(n)$ and the optimal convergence rate is $n^{-(s-1/2)}$. Hence when the sampling frame is localized,  the Casazza-Christensen method yields better convergence properties than the finite section method, even when $m$, the number of given samples, is smaller.
\end{rem}

\section{Constructing $S^{-1}$ for General Frames}\label{section:non-localized}
We now consider the case when the sampling frame $\{\psi_j\}_{j=1}^\infty$ is {\em not} well-localized, that is, \eqref{self-localization} is satisfied only for some $s\in (\frac{1}{2}, 1]$. Theorem \ref{thm-inv} demonstrates that the Casazza-Christensen method has low order convergence for $s > 1$ and the convergence rate does not hold for $s \le 1$  As discussed in Section \ref{sec:introduction}, effective numerical frame approximation techniques rely upon the accurate and efficient  approximation of $S^{-1}$, and in a variety of applications, for frames that are not well-localized.  Hence we introduce a new method of approximating the inverse frame operator $S^{-1}$ with better convergence properties. Our method is similar to \eqref{generalmethod}, but uses a projection onto a different finite-dimensional subspace that is generated by a well-localized frame. To this end, we introduce the concept of an admissible frame $\{\phi_j\}_{j=1}^\infty$ for $\cH$ which is defined as:
\begin{defn}\label{def_admissible}
A frame $\{\phi_j\}_{j=1}^\infty$ is admissible with respect to 
a frame $\{\psi_i\}_{i=1}^\infty$ if 
\begin{enumerate}
\item It is intrinsically self localized 
\begin{equation}\label{goodlocalization}
\left|\langle \phi_j, \phi_l \rangle\right| \leq c_0 (1+ |j-l|)^{-t}, \quad c>0, j,l\in \bN,
\end{equation}
with a localization rate $t>1$, and
\item We have 
\begin{equation}\label{mixedlocalization}
\left|\langle \psi_j, \phi_l \rangle\right| \leq c_1 (1+ |j-l|)^{-s}, \quad c>0, s>0, j,l\in \bN.
\end{equation}
\end{enumerate}
\end{defn}
We remark that for a frame to be admissible, we do not need $s > 1$, and in fact later we show that  $s>\frac{1}{2}$  ensures the convergence of the inverse frame operator. We also assume $t \ge s$. Otherwise, we can always take $\phi_j = \psi_j$.  

We now introduce some notation. For $n\in \bN$, let $\cG_n:=\spn\{\phi_j: 1\leq j\leq n\}$ and $Q_n$ be the projection from $\cH$ to $\cG_n$. Note that $Q_nS_m$ is an operator from $\cH$ to $\cG_n$, and we denote its restriction on $\cG_n$ by $W_n:=Q_nS_m\mid_{\cG_n}$.   The following operator is used to approximate $S^{-1}$:
\begin{equation}\label{newmethod}
W_n^{-1} Q_n f \rightarrow S^{-1}f, \quad \mbox{as } n\rightarrow \infty,
\end{equation}
where we have assumed that 
$W_n$ is an invertible operator on $\cG_n$.  Later we discuss the conditions under which this assumption holds.  The difference between (\ref{newmethod}) and  \eqref{generalmethod} is that here we use $Q_n$, the projection onto the finite-dimensional subspace $\cG_n$ generated by the admissible frame $\{\phi_j: 1\leq j\leq n\}$, instead of $P_n$, the projection onto the finite-dimensional subspace $\cH_n$ generated by the sampling frame $\{\psi_j: 1\leq j\leq n\}$.   This regularization allows for a numerically stable and convergent approximation of the inverse frame operator, even when the sampling is not done using well-localized frames. We will show that the convergence rate of approximating the inverse frame operator is now $t-1/2$ instead of $s-1/2$. Practically, when the sampling frame has a small localization rate $s$, we would like to find a frame with a greater localization rate $t$ that is admissible with respect to the sampling frame. 

To estimate the approximation error $\|S^{-1} f - W_n^{-1} Q_n f\|$, we first give its symbolic decomposition. Clearly
\begin{equation*}
S^{-1} f - W_n^{-1} Q_n f = S^{-1}(f - Q_n f) + S^{-1}( S - W_n ) W_n^{-1} Q_n f, 
\end{equation*}
and by the frame condition, \eqref{framebound}, we have $\|S^{-1}\| \leq 1/A$. It therefore follows that
\begin{equation}\label{inv-error-decomposition}
\left\|S^{-1} f - W_n^{-1} Q_n f \right\| \leq \frac{1}{A} \left\|f - Q_n f \right\| + \frac{1}{A} \left\|( S - W_n ) W_n^{-1} Q_n f \right\|. 
\end{equation}

We first estimate  $\|f - Q_n f\|$: 
\begin{prop}\label{prop-1sterror-nonlocal}
Assume that $\{\phi_j\}_{j = 1}^\infty$ is an admissible frame and that
\begin{equation}\label{framecoeff2}
\left|\langle f, \phi_j\rangle\right| \leq c j^{-t}, \quad c>0, j\in \bN.
\end{equation}
Then there exists a positive constant $c$ such that
\begin{equation*}
\|f-Q_n f \| \leq c n^{-(t-1/2)}.
\end{equation*}
\end{prop}
\begin{proof}
It is an immediate consequence of Lemma \ref{lem-Pn}.
\end{proof}

We shall next estimate the second term $\left\|( S - W_n ) W_n^{-1} Q_n f \right\|$ in \eqref{inv-error-decomposition} by first looking at $\|  S - W_n \|$. Note that the operator $S - W_n$ is restricted to $\cG_n$. Let $\mPhi_n: = \left[ \langle \phi_j, \phi_l \rangle \right]_{j,l=1}^n$ and $\lambda_{\min}(\mPhi_n)$ being its smallest eigenvalue. We here assume $\mPhi_n$ is invertible and $\lambda_{\min}(\mPhi_n)>0$. Otherwise, we can use its invertible principal submatrix instead. 
\begin{lem}\label{lem-S-Rn}
Define \begin{equation}\label{Bmn}
B_{m,n}:= \frac{1}{\lambda_{\min}(\mPhi_n)}\sum_{j=m+1}^\infty \sum_{l=1}^n |\langle \phi_l, \psi_j \rangle|^2,
\end{equation}
and choose $m > n$.
Then
\begin{equation*}
\|  S - W_n \| \leq B_{m,n}.
\end{equation*}
\end{lem}
\begin{proof}
For any $g\in \cG_n$, since $Q_n$ is the projection onto $\cG_n$, we have $Q_n g =g$. Recall that $W_n$ is the restriction of $Q_n S_m$ on $\cG_n$. It follows that for any $g\in \cG_n$
\begin{equation*}
\langle W_n g, g \rangle = \langle Q_n S_m g, g \rangle = \langle S_m g, Q_n g \rangle = \langle S_m g, g \rangle,
\end{equation*}
which implies that
\begin{equation*}
\left\langle (S - W_n) g, g  \right\rangle = \left\langle (S - S_m) g, g \right\rangle = \sum_{j=m+1}^\infty \left|\langle g, \psi_j \rangle \right|^2.
\end{equation*}
Since $g\in \cG_n$, we can write $g=\sum_{l=1}^n a_l \phi_l$ for some $\va \in \bR^n$. It follows that
\begin{equation*}
\langle (S - W_n) g, g \rangle = \sum_{j=m+1}^\infty \biggl| \sum_{l=1}^n a_l \langle \phi_l, \psi_j \rangle \biggr|^2 \leq \| \va\|^2 \sum_{j=m+1}^\infty \sum_{l=1}^n |\langle \phi_l, \psi_j \rangle|^2.
\end{equation*}
Note that $\| g\|^2 = \va^T \mPhi_n \va \geq \lambda_{\min}(\mPhi_n) \|\va \|^2$. Substituting back into the above inequality yields $\langle (S - W_n) g, g \rangle \leq B_{m,n} \| g\|^2$, implying the desired result.
\end{proof}

We next give an estimate of $\| W_n^{-1}\|$ also depending on $B_{m,n}$.
\begin{lem}\label{lem-Rn-inv}
If $B_{m,n}<A$, then $\{Q_n \psi_j\}_{j=1}^m$ is a frame for $\cG_n$ with frame bounds $A - B_{m,n}$, $B$ and the frame operator $W_n$. Moreover,
\begin{equation*}
\| W_n^{-1}\| \leq \frac{1}{A- B_{m,n}}.
\end{equation*}
\end{lem}
\begin{proof}
We will check the frame condition \eqref{framebound} by direct calculation. For any $g\in \cG_n$, we have $Q_n g=g$, which implies that
\begin{equation*}
\sum_{j=1}^m |\langle g, Q_n \psi_j \rangle|^2 = \sum_{j=1}^m |\langle Q_n g,  \psi_j \rangle|^2 = \sum_{j=1}^m |\langle g, \psi_j \rangle|^2.
\end{equation*}
To see the upper bound, observe that
\begin{equation*}
\sum_{j=1}^m |\langle g, \psi_j \rangle|^2 \leq \sum_{j=1}^\infty |\langle g, \psi_j \rangle|^2  \leq B \| g\|^2.
\end{equation*}
To show the lower bound, by Lemma \ref{lem-S-Rn}, we have $\sum_{j=m+1}^\infty |\langle g, \psi_j \rangle|^2 \leq B_{m,n} \| g\|^2$. It follows that 
\begin{equation*}
\sum_{j=1}^m |\langle g, \psi_j \rangle|^2 = \sum_{j=1}^\infty |\langle g, \psi_j \rangle|^2 - \sum_{j=m+1}^\infty |\langle g, \psi_j \rangle|^2 \geq (A-B_{m,n}) \| g\|^2.
\end{equation*}
 The above two inequalities implies that $\{Q_n \psi_j\}_{j=1}^m$ is a frame for $\cG_n$ with frame bounds $A - B_{m,n}$, $B$ if $B_{m,n} <A$.

To show $W_n$ is the associated frame operator, observe that for any $g\in \cG_n $,
\begin{equation*}
\sum_{j=1}^m \langle g, Q_n \psi_j \rangle Q_n\psi_j = \sum_{j=1}^m \langle Q_n g, \psi_j \rangle Q_n\psi_j = \sum_{j=1}^m \langle g, \psi_j \rangle Q_n\psi_j= Q_nS_m g = W_n g.
\end{equation*}

The bound of $\|W_n^{-1}\|$ is an immediate result from the lower frame bound of $W_n$.
\end{proof}

Lemmas \ref{lem-S-Rn} and \ref{lem-Rn-inv} yield corresponding estimates for $\|S-W_n\|$ and $\|W_n^{-1}\|$ dependent on the constant $B_{m,n}$. We now estimate $B_{m,n}$ under the admissibility assumption.

\begin{lem}\label{lem-Bmn}
Let $\{\phi_l\}_{l = 1}^\infty$ be admissible with respect to the frame
 $\{\psi_j\}_{j = 1}^\infty$  where \eqref{mixedlocalization} holds with $s>1/2$.  Then
\begin{equation*}
B_{m,n} \leq \frac{c_1^2}{2s-1} \frac{1}{\lambda_{\min}(\mPhi_n)} n (m-n)^{-(2s-1)}.
\end{equation*}
\end{lem}
\begin{proof}
By \eqref{mixedlocalization}, we have
\begin{equation*}
\sum_{j=m+1}^\infty \sum_{l=1}^n |\langle \phi_l, \psi_j \rangle|^2 \leq \sum_{j=m+1}^\infty \sum_{l=1}^n c_1^2 (1+|j-l|)^{-2s} .
\end{equation*}
It follows from $m>n$  and $s>1/2$ that
\begin{equation*}
\sum_{j=m+1}^\infty \sum_{l=1}^n |\langle \phi_l, \psi_j \rangle|^2 \leq c_1^2 \sum_{j=m+1}^\infty n (1+j-n)^{-2s}  \leq \frac{c_1^2}{2s-1} n (m-n)^{-(2s-1)}.
\end{equation*}
Substituting into \eqref{Bmn} yields the desired result.
\end{proof}

We now present an estimate of $\|( S - W_n ) W_n^{-1} Q_n f\|$ by combining the above three lemmas:

\begin{prop}\label{prop-2nderror-nonlocal}
Suppose $\{\phi_j\}_{j=1}^\infty$ is admissible with respect to the frame
 $\{\psi_j\}_{j = 1}^\infty$ and that  \eqref{framecoeff2} holds. If 
\begin{equation}\label{choose-m-nonlocal}
m= n + \alpha \biggl[\frac{2 c_1^2}{A(2s-1) \lambda_{\min}(\mPhi_n)} \biggr]^{\frac{1}{2s-1}} n^ \frac{t+ 1/2}{2s-1}, \quad \alpha>0,
\end{equation}
then $B_{m,n} \leq \alpha^{-(2s-1)}n^{-(t-1/2)}$. Moreover, there exists a positive constant $c$ such that for all $n\in \bN$
\begin{equation*}
\|( S - W_n ) W_n^{-1} Q_n f\| \leq c n^{-(t-1/2)}.
\end{equation*}
\end{prop}
\begin{proof}
The bound of $B_{m,n}$ follows from substituting \eqref{choose-m-nonlocal} into the estimate of  $B_{m,n}$ in Lemma \ref{lem-Bmn}. It follows from Lemmas \ref{lem-S-Rn} and \ref{lem-Rn-inv} that
\begin{equation*}
\|( S - W_n ) W_n^{-1} Q_n f\| \leq \|S - W_n \| \|W_n^{-1}\| \|Q_n \| \| f\| \leq \frac{B_{m,n} }{A-B_{m,n}} \| f\|.
\end{equation*}
This combined with the bound of $B_{m,n}$ implies the desired result.
\end{proof}

We remark that in the above proposition, $m$ is chosen to obtain the optimal convergence rate $n^{-(t-1/2)}$.

The main theorem regarding the convergence rate for our new method of approximating the inverse frame operator $S^{-1}$, (\ref{newmethod}), as a summarized result of the two estimates in Propositions \ref{prop-1sterror-nonlocal} and \ref{prop-2nderror-nonlocal} can now be given:

\begin{thm}\label{thm-inverse-frame}
Let $\{\phi_j\}_{j = 1}^\infty$ be an admissible frame with respect to the frame
 $\{\psi_j\}_{j = 1}^\infty$ and assume that
\eqref{framecoeff2} holds. If $m$ is chosen as in \eqref{choose-m-nonlocal}, then there exists a positive constant $c$ such that
\begin{equation*}
\|S^{-1}f - W_n^{-1}Q_n f\| \leq c n^{-(t-1/2)}.
\end{equation*}
\end{thm}
\begin{proof}
The desired result follows directly from substituting estimates in Propositions \ref{prop-1sterror-nonlocal} and \ref{prop-2nderror-nonlocal} into the error decomposition in \eqref{inv-error-decomposition}.
\end{proof}

\section{Numerical Frame Approximation}\label{section:reconstruction}
In this section we employ the approximation of inverse frame operator $S^{-1}$ presented in (\ref{newmethod}) to obtain an efficient reconstruction of an unknown function $f$ in $\cH$ from the sampling data $\{\langle f, \psi_j \rangle\}_{j=1}^m$. 
Recall that $f$ can be represented 
by (\ref{framesum}).
However, we typically only have access to finite sampling data $\{\langle f, \psi_j \rangle\}_{j=1}^m$, and  moreover, we do not have a closed form for $S^{-1}$.  Hence we will utilize the approximation $W_n^{-1}Q_n$ in (\ref{newmethod}) and reconstruct $f$ as 
\begin{equation}\label{reconstruction}
f_{n,m}: = \sum_{j=1}^m \langle f, \psi_j \rangle W_n^{-1}Q_n \psi_j
= W_n^{-1}Q_n S_m f. 
\end{equation}  
We remark that since $W_n$ is the restriction of $Q_n S_m$ on $\cG_n$, the restriction of $W_n^{-1}Q_n S_m$ on $\cG_n$ is the same as the identity operator. Therefore (\ref{reconstruction}) is exact for $f$ in $\cG_n$. Furthermore, if the finite-dimensional subspace $\cG_n$ is ``close'' to the underlying space $\cH$, the reconstructed function should also be ``close" to the unknown function $f$. Recall that $\cG_n$ is generated by the admissible frame $\{\phi_j\}_{j=1}^n$,  ensuring good approximation properties, as determined by Theorem \ref{thm-inverse-frame} for suitable sampling space size $m$.\footnote{Clearly the convergence of (\ref{reconstruction}) will depend on the smoothness properties of $f$ with respect to the admissible frame $\phi$.  In the case where $f$ is only piecewise smooth, we note that post-processing can be applied using the spectral reprojection method, \cite{GelbHines}.}

To estimate the approximation error $\|f-f_{n,m}\|$, we first note that
\begin{equation*}
f-f_{n,m} = f- Q_n f + Q_n f - W_n^{-1} Q_n S_m f.
\end{equation*}
Since $Q_nf \in \cG_n$, we have $Q_nf = W_n^{-1} Q_n S_m Q_n f$. Hence
\begin{equation*}
\| f-f_{n,m}\| \leq \| f- Q_n f\| + \|W_n^{-1} Q_n S_m\| \|f - Q_nf \|.
\end{equation*}
Since $\| Q_n\|\leq 1$ and $\|S_m\| \leq \|S\| \leq B$ by the frame condition (\ref{framebound}), we have
\begin{equation}\label{error-decomposition-reconstruction}
\| f-f_{n,m}\| \leq \| f- Q_n f\| + B\|W_n^{-1}\| \|f - Q_nf \|.
\end{equation}
Note that an estimate of $\|f - Q_nf \|$ was provided in Proposition \ref{prop-1sterror-nonlocal}. It remains to estimate $\|W_n^{-1} \|$. In fact, Theorem \ref{thm-reconstruction} shows that we can choose $m$ depending on $n$ such that $\|W_n^{-1}\|$ is uniformly bounded for all $n$.  By Lemma \ref{lem-Rn-inv}, it suffices to choose $m$ such that $A_{m,n} \leq A/2$ in (\ref{Amn}) for all $n\in \bN$, where $A$ is the upper frame bound in (\ref{framebound}). 

\begin{thm}\label{thm-reconstruction}
Suppose the assumption \eqref{mixedlocalization} holds with $s>\frac{1}{2}$. If we let 
\begin{equation}\label{choose-m-reconstruction}
m= n + \biggl[\frac{2 c_1^2 n}{A(2s-1) \lambda_{\min}(\mPhi_n)} \biggr]^{\frac{1}{2s-1}},
\end{equation}
then $\|W_n^{-1} \|\leq \frac{2}{A}$ for all $n\in \bN$. Furthermore, if assumptions \eqref{goodlocalization} and \eqref{framecoeff2} also hold, then there exists a positive constant $c$ such that
\begin{equation*}
\|f-f_{n,m}\| \leq c n^{-(t-1/2)}.
\end{equation*}
\end{thm}
\begin{proof}
The bound of $\| W_n^{-1}\|$ follows immediately from Lemma \ref{lem-Rn-inv} and substituting \eqref{choose-m-reconstruction} into \eqref{Bmn}. 

The estimate of $\|f-f_{n,m}\|$ follows from substituting the bound of $\|W_n^{-1}\|$ and the estimate of $\|f - Q_nf \|$ in Proposition \ref{prop-1sterror-nonlocal} into the error decomposition \eqref{error-decomposition-reconstruction}.
\end{proof}

We make the following remarks about the choice of $m$:
\begin{rem}\label{rem:remark3}
Notice that  $m$ in \eqref{choose-m-reconstruction} is much smaller than $m$ in \eqref{choose-m-nonlocal}. This is because  we only need $B_{m,n}\leq A/2$ to obtain the optimal order for reconstructing the unknown function $f$, while  $B_{m,n}={\mathcal O}(n^{-(t-1/2)})$ is required to obtain optimal order for approximating the inverse frame operator.  From a practical point of view, it is only necessary to satisfy (\ref{choose-m-reconstruction}) for function reconstruction.
\end{rem}
\begin{rem}\label{rem:remark4}
When $\{\phi_j\}_{j=1}^\infty$ is a Riesz basis, the minimal eigenvalue $\lambda_{\min}(\mPhi_n)$ of $\mPhi_n$ is bounded below for all $n\in \bN$.  To ensure $\|W_n^{-1}\|$ is uniformly bounded in that case, for $s \in (1/2,1)$ we have $m =  {\mathcal O}(n^{\frac{1}{2s-1}})$, and for $s > 1$  we have $m = {\mathcal O}(n)$.
\end{rem}

Finally, Proposition \ref{prop-ls} shows that $f_{n,m}$ is the least squares solution for $\{\langle f, \psi_j\rangle\}_{j=1}^m$ in $\cG_n$.

\begin{prop}\label{prop-ls}
Suppose $\{\phi_j\}_{j = 1}^n$ is admissible with $s>\frac{1}{2}$ and $m$ is chosen as in \eqref{choose-m-reconstruction}. Then
\begin{equation*}
f_{n,m} = \mathop{\argmin}\limits_{g\in\cG_n} \sum_{j=1}^m \left| \langle g, \psi_j\rangle -\langle f, \psi_j\rangle \right|^2.
\end{equation*} 
\end{prop}
\begin{proof}
We first reformulate the least squares problem in terms of the coefficients. For any $g=\sum_{j=1}^na_j \phi_j \in\cG_n$, we have
\begin{equation}\label{leastsquares}
\sum_{j=1}^m \left| \langle g, \psi_j\rangle -\langle f, \psi_j\rangle \right|^2 = \left\|\mOmega\va - \vhf\right\|^2,
\end{equation}
where $\mOmega:=[\langle\psi_j, \phi_l\rangle]_{j,l=1}^{m,n}$, $\va:=[a_j]_{j=1}^n$ and $\vhf=[\langle f, \psi_j\rangle]_{j=1}^m$.

To show that $f_{n,m}$ is the least squares solution of (\ref{leastsquares}),  we note that \eqref{mixedlocalization} and choosing $m$ to satisfy \eqref{choose-m-reconstruction} ensure that $W_n$ is invertible and $f_{n,m}$ is well defined. Thus for $f_{n,m}=\sum_{j=1}^nc_j \phi_j \in\cG_n$,  it suffices to show
\begin{equation}\label{prop-ls-1}
\vc =\mathop{\argmin}\limits_{\va\in\bR_n} \left\|\mOmega\va - \vhf\right\|^2.
\end{equation}
To demonstrate (\ref{prop-ls-1}), first observe from \eqref{reconstruction} that $W_nf_{n,m}= Q_n S_mf$. Since $W_n$ is the restriction of $Q_nS_m$ on $\cG_n$, we have $Q_nS_mf_{n,m} = Q_n S_mf$. Furthermore, since $Q_n$ is the projection onto $\cG_n$, we have
\begin{equation*}
\langle S_mf_{n,m} , \phi_j\rangle = \langle S_mf, \phi_j\rangle, \quad 1\leq j\leq n.
\end{equation*}
A direct calculation from the above equalities yields that 
\begin{equation*}
\mOmega^T\mOmega \vc = \mOmega^T \vhf,
\end{equation*}
which are the normal equations for \eqref{prop-ls-1}.
\end{proof}

\subsection{Computational Algorithms for $f_{n,m}$}
\label{sec:computational}
We now discuss some algorithms for computing $f_{n,m}:=W_n^{-1}Q_nS_mf$.  Calculating $S_mf$ from the sampling data $\{ \langle f, \psi_j \rangle \}_{j=1}^m$ is straightforward using (\ref{finiteframe}).  To calculate $Q_n S_m f$, we introduce the operator 
\begin{equation*}
U_n(g):=\sum_{j=1}^n \langle g, \phi_j \rangle \phi_j, \quad g\in \cG_n.
\end{equation*}
Note from (\ref{finiteframe}) that $U_n$ is the frame operator for the finite frame $\{\phi_j\}_{j=1}^n$ in $\cG_n$ and the projection $Q_n g$ for any $g\in \cG_n$ can therefore be computed by
\begin{equation*}
Q_n g = \sum_{j=1}^n \langle g, \phi_j \rangle U_n^{-1}\phi_j.
\end{equation*}
Typically the inverse frame operator $U_n^{-1}$ is determined by solving the finite system  $U_n{\bf g} = {\bf \phi}$ for ${\bf g}$.  It follows that
\begin{equation}\label{PnSmf}
W_nf_{n,m} = Q_nS_m f=  \sum_{j=1}^n \langle S_mf, \phi_j \rangle U_n^{-1}\phi_j = \sum_{j=1}^n \sum_{l=1}^m \langle f, \psi_l \rangle \langle \psi_l, \phi_j \rangle U_n^{-1}\phi_j.
\end{equation}
Hence to obtain $f_{n,m}$, we need to apply $W_n^{-1}$ to (\ref{PnSmf}). Recall that $W_n$ is the restriction of $Q_n S_m$ on $\cG_n$. When $m$ is chosen as in \eqref{choose-m-reconstruction}, the operator $W_n$ has lower bound $A/2$ and upper bound $B$.   To compute $W_n^{-1}$, we apply the iterative algorithm given in \cite{MR1162107}:
\begin{eqnarray}\label{iterativealg}
f_{n,m}^{(0)} &=& 0\nonumber\\
f_{n,m}^{(j)} &=& f_{n,m}^{(j-1)} + \frac{2}{A/2 + B} W_n (f_{n,m} - f_{n,m}^{(j-1)}), \quad j\in \bN.
\end{eqnarray}
To determine the convergence of (\ref{iterativealg}), observe that
\begin{equation*}
f_{n,m} -f_{n,m}^{(j)} = \left( I - \frac{2}{A/2 + B} W_n \right) (f_{n,m} - f_{n,m}^{(j-1)})
= \left( I - \frac{2}{A/2 + B} W_n \right)^j f_{n,m},
\end{equation*}
and since $\left\|I - \frac{2}{A/2 + B} W_n\right\| \leq \frac{B-A/2}{A/2 + B}$,
 it follows that
\begin{equation}\label{conv-direct}
\|f_{n,m} -f_{n,m}^{(j)}\| \leq \left(\frac{B-A/2}{A/2 + B} \right)^j \|f_{n,m}\|.
\end{equation}

Unfortunately (\ref{iterativealg}) requires explicit estimates of the frame bounds $A$ and $B$ that are unknown or impractical to obtain in most cases. Moreover, the convergence of this iteration method is quite slow if $B/A$ is large. Hence we employ the {\it conjugate gradient acceleration} method introduced in \cite{Grochenig1993} to compute $f_{n,m}$.

\begin{algorithm}\label{conjgradient}(Conjugate gradient acceleration method for computing $f_{n,m}$)
\begin{enumerate}
\item  Initialization: $f_{n,m}^{(0)}=0$, \quad $r_0=p_0=W_n f_{n,m}$, \quad $p_{-1}=0$
\item {\bf repeat}($j\geq 0$)
\item \qquad $\alpha_j = \frac{\langle r_j, p_j \rangle}{\langle p_j, W_n p_j \rangle}$
\item \qquad $f_{n,m}^{(j+1)} = f_{n,m}^{(j)} + \alpha_j p_j$
\item \qquad $r_{j+1} = r_j - \alpha_j W_n p_j$
\item \qquad $p_{j+1}= W_n p_j -\frac{\langle W_n p_j, W_n p_j \rangle}{\langle p_j, W_n p_j \rangle} p_j - \frac{\langle W_n p_j, W_n p_{j-1} \rangle}{\langle p_{j-1}, W_n p_{j-1} \rangle} p_{j-1}$, with the last term set to zero when
$p_{j-1}=0$.
\item {\bf until} the stopping criterion is met
\end{enumerate}
\end{algorithm}
The following convergence result for Algorithm \ref{conjgradient} is shown in \cite{Grochenig1993}:

\begin{prop}
Let $f_{n,m}^{(j)}$ be computed by Algorithm \ref{conjgradient}. There holds that for all $j\in \bN$
\begin{equation*}
\|f_{n,m} -f_{n,m}^{(j)}\| \leq \frac{1+\sigma}{1-\sigma}\frac{2\sigma^j}{1 + \sigma^{2j}} \|f_{n,m}\|,
\end{equation*}
where $\sigma= \left(\sqrt{B} - \sqrt{A/2}\right) / \left(\sqrt{B} + \sqrt{A/2}\right)$.
\end{prop}

We remark that Algorithm \ref{conjgradient} does not require explicit estimates of the frame bounds. It also improves the convergence rate of \eqref{conv-direct}.

\subsection{Sampling with Fourier Frames}\label{section:Fourierframe}
Motivated by the data acquisition techniques used in magnetic resonance imaging (MRI), \cite{BW00, GelbHines, Aditya09}, in this section we consider the special case of
sampling with Fourier frames. Specifically, we define $\cH=L^2[-1,1]$ and let 
\begin{equation}\label{FourierFrame}
\psi_j (x) = \se^{- \si\pi \lambda_j x}, \quad \lambda_j = j + \xi_j,\quad j\in \bZ,
\end{equation}
where each $\xi_j$ is a random variable uniformly distributed in $[-1/4, 1/4]$. The sampling scheme (\ref{FourierFrame}) describes the situation where Fourier data is collected mechanically, and the samples are ``jittered'' from the presumably uniform distribution.  By Kadec's 1/4-Theorem \cite{Christensen2003}, $\{\psi_j\}_{j=1}^\infty$ is a frame in $\cH$. For ease of presentation, below  we use $\bZ$ as the index set, and note that the results are similarly obtained to those from previous sections using $\bN$. 

Suppose we are given the first $2m + 1$ frame coefficients $\{\langle f, \psi_j \rangle\}_{j=-m}^m$ for an unknown function $f \in \cH$.
We will reconstruct the function $f$ with the Fourier basis. That is, we let 
\begin{equation}\label{FourierBasis}
\phi_j = \se^{- \si\pi j x},\, j\in \bZ.
\end{equation} 
The Fourier basis (\ref{FourierBasis}) is an admissible frame with respect to the Fourier frame given by (\ref{FourierFrame}):
\begin{lem}\label{lem-FourierFrame}
Suppose $\psi_j$ and $\phi_j$, $j \in \bZ$ are given in \eqref{FourierFrame} and $\eqref{FourierBasis}$ respectively. Then $\{\phi_j\}_{j\in\bZ}$ is admissible for all $t>0$ and  $\eqref{mixedlocalization}$ holds with $c_1=8/\pi$ and $s=1$.
\end{lem}
\begin{proof}
Since $\{\phi_j\}_{j\in\bZ}$ is orthonormal, $\eqref{goodlocalization}$ holds for all $t>0$.
To check $\eqref{mixedlocalization}$, we have for $j,l\in \bZ$:
\begin{equation*}
\left| \langle \psi_j, \phi_l\rangle \right|= \left| \int_{-1}^1 \se^{-\si\pi (\lambda_j -l)x}dx \right| = \left| \frac{2}{\pi (\lambda_j - l)} \sin\left(\pi(\lambda_j -l) \right) \right| .
\end{equation*}
When $j=l$, $|\lambda_j -l| = |\xi_j| \leq 1/4$ and $\left| \langle \psi_j, \phi_l\rangle \right| \leq 2$. When $j\neq l$, $|\lambda_j -l| = | j + \xi_j - l| \geq \frac{1}{4} (1+|j-l|)$ and $\left| \langle \psi_j, \phi_l\rangle \right| \leq \frac{2}{\pi} |\lambda_j -l|^{-1} \leq \frac{8}{\pi} (1+|j-l|)^{-1}$. It follows that $\eqref{mixedlocalization}$ holds with $c_1=8/\pi$ and $s=1$.
\end{proof}

Note that when $\{\phi_j\}_{j\in \bZ}$ is the Fourier basis, the matrix $\mPhi_n$ is the identity matrix. The $m$ in \eqref{choose-m-reconstruction} reduces to ${\mathcal O}(n)$. We have the following result.

\begin{prop}
Suppose $\{\psi_j\}_{j\in \bZ}$ and $\{\phi_j\}_{j\in \bZ}$ are given in \eqref{FourierFrame} and $\eqref{FourierBasis}$ respectively. If \eqref{framecoeff} holds and $m= \frac{A\pi^2+ 128 }{A\pi^2} n$, then
\begin{equation*}
\|f-f_{n,m}\| \leq c n^{-(t-1/2)}.
\end{equation*}
\end{prop}
\begin{proof}
It is an immediate consequence of Lemma \ref{lem-FourierFrame} and Theorem \ref{thm-reconstruction}.
\end{proof}

We now present some numerical experiments to illustrate our results.
\begin{example}\label{example1}
$$f(x) = \se^{-x^2}$$
\end{example}
For each $2n+1$ frame elements used to reconstruct the function in \eqref{reconstruction}, we set $m=1.4 n$, where $2m+1$ is the number of given Fourier frame coefficients with the Fourier frame defined in \eqref{FourierFrame}.  We compare the numerical results of our method \eqref{reconstruction} with the Casazza-Christensen method. In both cases we employ the conjugate gradient acceleration method in Algorithm \ref{conjgradient} to construct the inverse frame operator with stopping criterion of relative error  less than 1.E-5.  
We also compare our results to the standard Fourier reconstruction, or $\lambda_j = j$ in (\ref{FourierFrame}) and $m = n$.  Note that in this case the frame operator is self-dual with $S = S^{-1} = I$.  
Table \ref{table:smooth1} compares the $L_2$ error, computational cost, and the condition number of these schemes for various $n$,  
and demonstrates that our method \eqref{reconstruction} converges more quickly with fewer iterations and has better conditioning than the Casazza-Christensen method. 
As is also evident from Table \ref{table:smooth1}, our new method provides the same rate of convergence as the standard Fourier partial sum.  In fact, as the smoothness of the target function increases, it is apparent that our numerical frame approximation yields the same {\em exponential} convergence properties as the harmonic Fourier approximation does.
Figure \ref{figure:smooth1} compares the function reconstructions and corresponding point-wise errors.  
 
\begin{table}[htbp]
\centering
\begin{tabular}{r| c c c|c c | c c}
\multirow{2}{*}{$n$} & \multicolumn{3}{c|}{error} & \multicolumn{2}{c|}{iterations} & \multicolumn{2}{c}{condition number}\\
 & (a) & (b) & (c) & (a) & (b) & (a) & (b)\\
\hline
16     & 4.6E-2&  1.4E-3&  1.4E-3       &  20 & 12  &  23.4 &  4.6 \\
 32    &  2.1E-2&  6.0E-4&  6.0E-4  &  24  &  12  &  23.8 &  4.2 \\
 64   &  1.2E-2  &  2.6E-4 &  2.6E-4    &  25  &  12 &  23.9  &  4.5\\
 128   &  9.2E-3  & 1.3E-4 & 1.3E-4   &  28  &  13  &  28.8  &  5.4\\
 256  &  7.6E-3&  6.0E-5&  6.0E-5   &  29 &  13   &  33.4 &  5.8 \\
\end{tabular}
\caption{Results using  (a) the Casazza-Christensen method, (b) our new method \eqref{reconstruction} with $m = 1.4n$ and (c) the standard Fourier reconstruction method for Example \ref{example1} . }\label{table:smooth1}
\end{table}

\begin{figure}[htbp]
\centering
\subfloat[]{\label{1a}\includegraphics[scale=0.2]{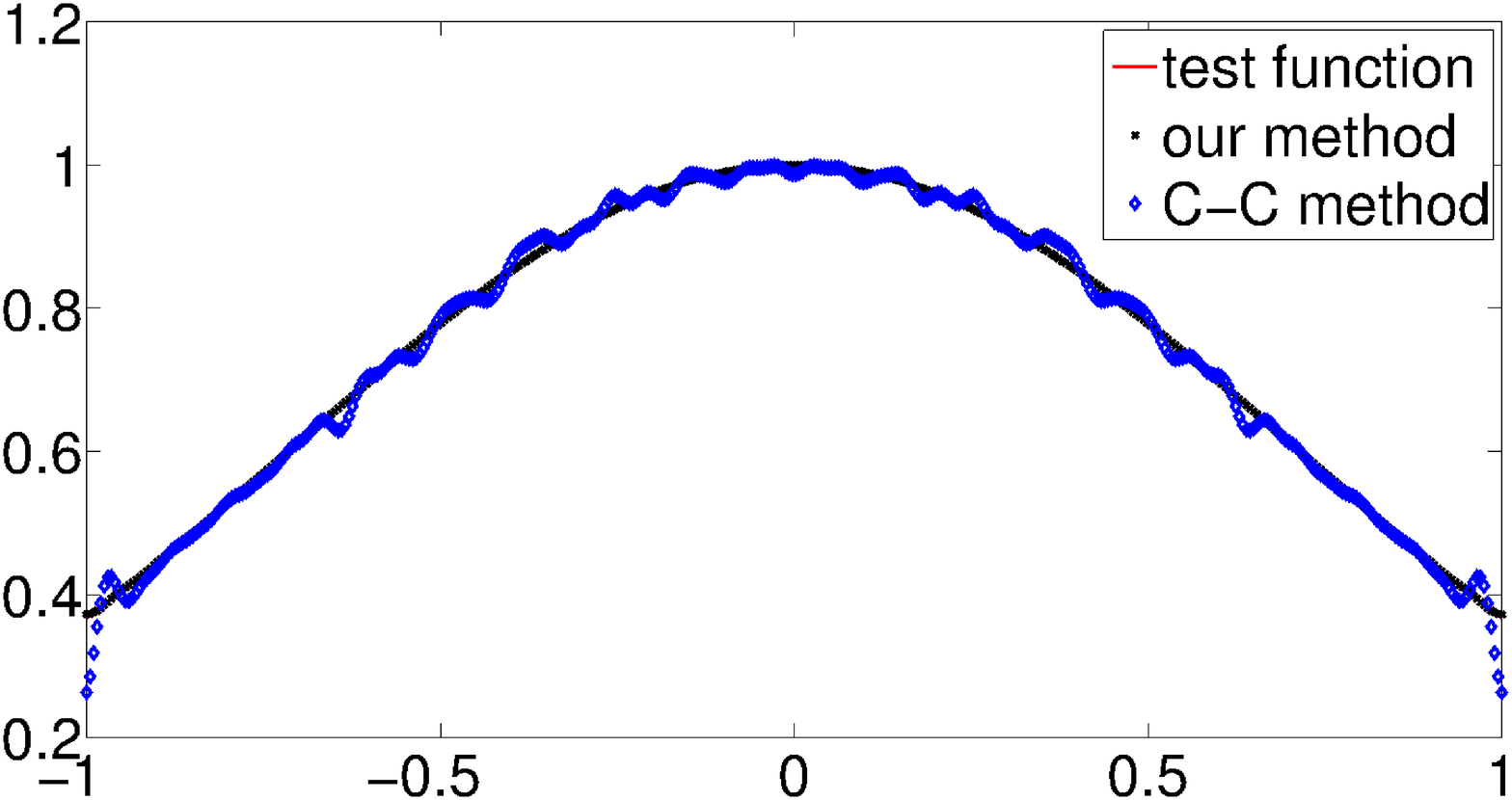}}\\
\subfloat[]{\label{1b}\includegraphics[scale=0.2]{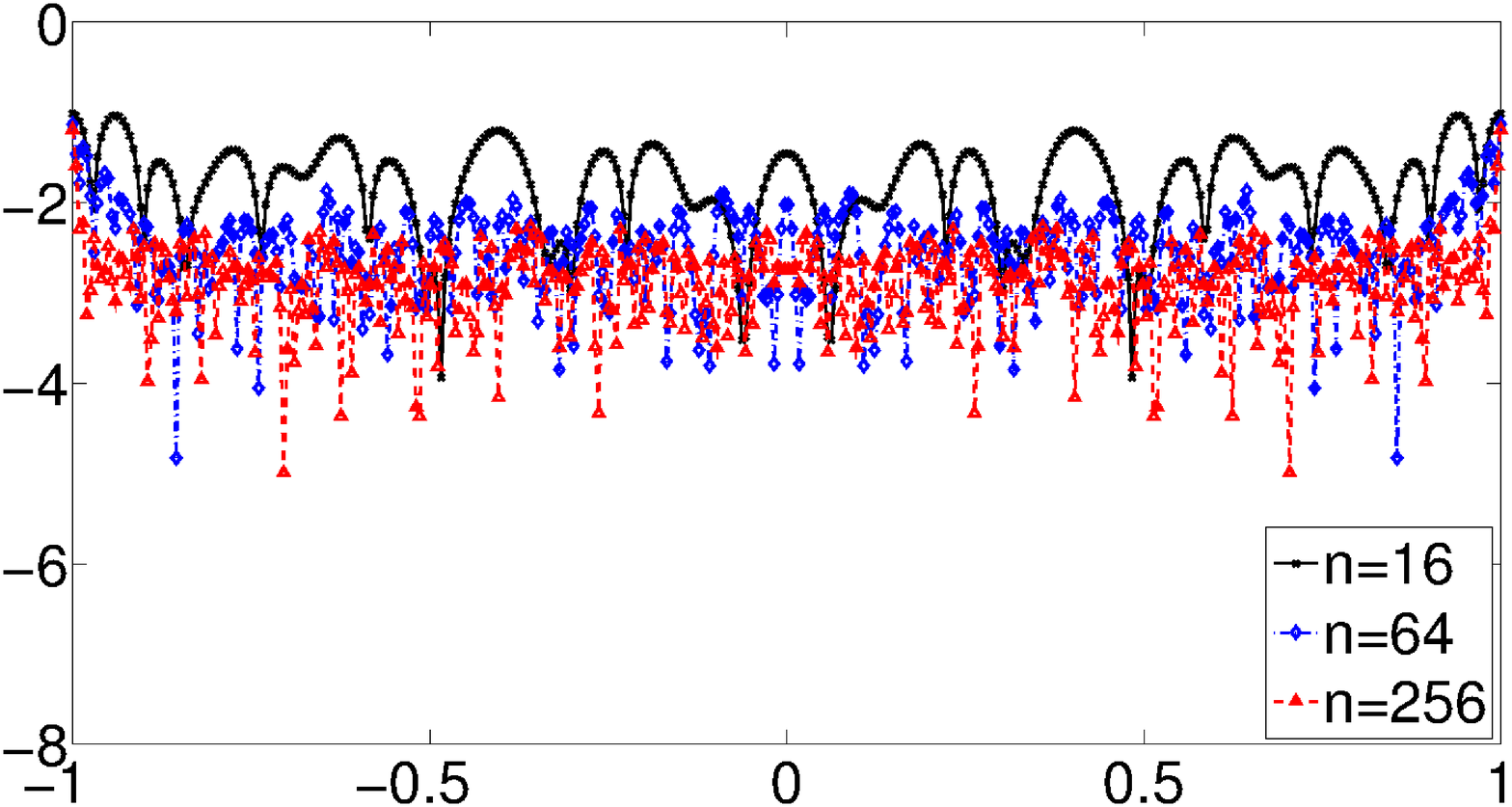}}
\subfloat[]{\label{1c}\includegraphics[scale=0.2]{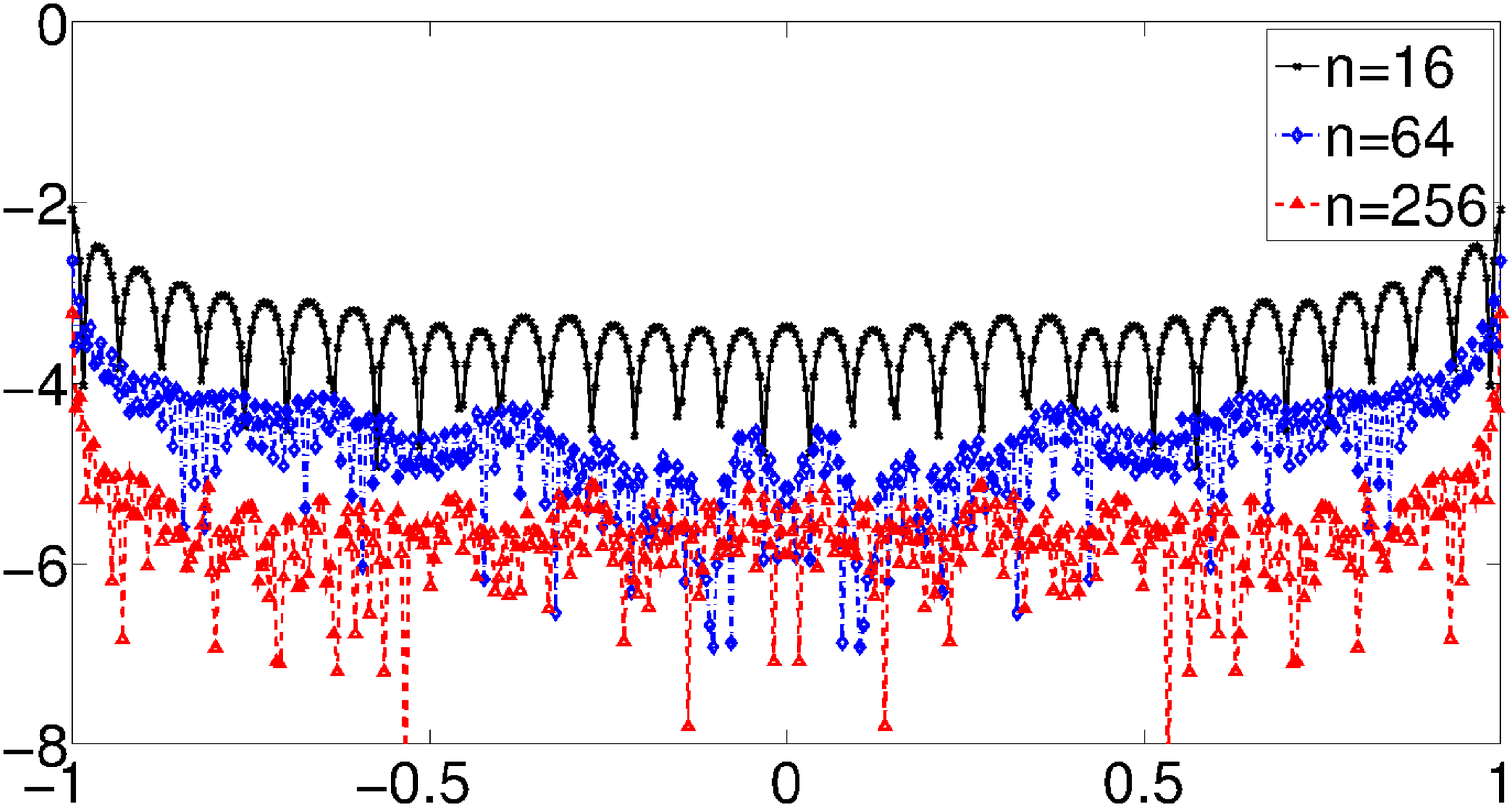}}
\caption{(a) Reconstruction for Example \ref{example1}; (b) Point-wise $log_{10}$ error of the reconstruction by the Casazza-Christensen method; (c) Point-wise $log_{10}$ error of the reconstruction by our method \eqref{reconstruction}.}
\label{figure:smooth1}
\end{figure}


\begin{example}\label{example2}
$$f(x) = \cos^3(\pi x) (\sin^2(x) +1 )$$
\end{example}
Here we assume we are given fewer frame coefficients, $m = 1.2n$.  Table \ref{table:smooth2} compares the $L_2$ error, computational cost, and the condition number of our method (\ref{reconstruction}) to the Casazza-Christensen method, and also displays the standard Fourier approximation error.  Figure \ref{figure:smooth2} displays the reconstructions and point-wise errors for each method. 
\begin{table}[htbp]
\centering
\begin{tabular}{r| c c c|c c | c c}
\multirow{2}{*}{$n$} & \multicolumn{3}{c|}{error} & \multicolumn{2}{c|}{iterations} & \multicolumn{2}{c}{condition number}\\
 & (a) & (b) & (c) & (a) & (b) & (a) & (b)\\
\hline
16    & 2.6E-2  &  1.8E-3  &  1.8E-3     &  21  & 12  &  19.9  &  4.5\\
32  & 1.0E-2  &  7.4E-4   &  7.3E-4     &  22  &  12 &  20.9 &  4.5 \\
64  & 2.2E-2 &  3.2E-4  &  3.2E-4    &  27  &  13  &  28.3  &  5.4\\
128  & 1.8E-3 &  1.6E-4  &  1.6E-4    &  30  &  13  &  30.5  &  5.5\\
256 & 4.7E-3  &  7.3E-5  &  7.3E-5   &  30  &  13  &  32.8  &  5.7\\
\end{tabular}
\caption{Comparison of (a) the Casazza-Christensen method, (b) our new method (\ref{reconstruction}), and (c) the standard Fourier reconstruction method for Example \ref{example2}. }\label{table:smooth2}
\end{table}

\begin{figure}[htbp]
\centering
\subfloat[]{\label{2a}\includegraphics[scale=0.2]{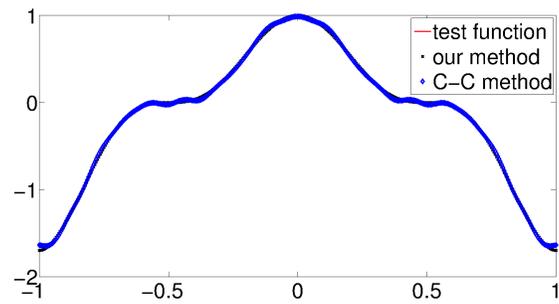}}\\
\subfloat[]{\label{2b}\includegraphics[scale=0.2]{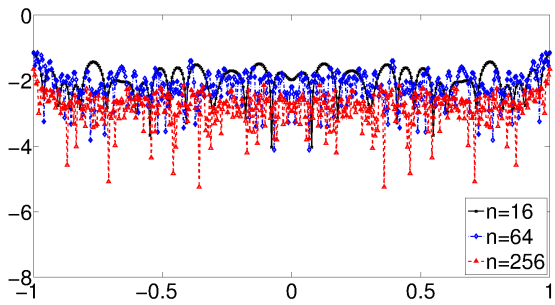}}
\subfloat[]{\label{2c}\includegraphics[scale=0.2]{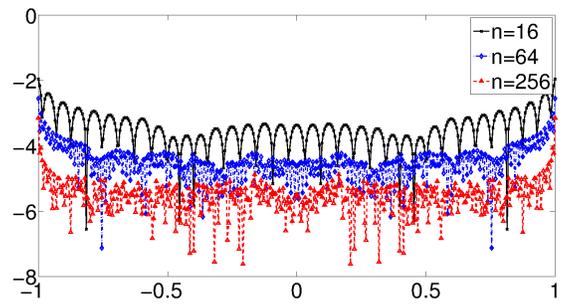}}
\caption{(a) Reconstruction of $f(x)=\cos^3(\pi x) (\sin^2(x) +1 )$; (b) Point-wise $log_{10}$ error for the Casazza-Christensen method; (c) Point-wise $log_{10}$ error for (\ref{reconstruction}); }
\label{figure:smooth2}
\end{figure}

Once again, as illustrated in Table \ref{table:smooth2} and Figure \ref{figure:smooth2}, we see that our method (\ref{reconstruction}) converges even when less sampling data ($m=1.2n$) is used, and its numerical properties are better than for the Casazza-Christensen method.  It appears as though since the Fourier frame is {\em not} well localized, the convergence rate analysis in  Section \ref{section:localized} does not apply.

\begin{example}\label{example3}
$$f(x) = (1-x^2)^3.$$
\end{example}
Example \ref{example3} provides a smoother test case.  Table \ref{table:smooth3} and Figure \ref{figure:smooth3} compare the results using our method with those from the Casazza-Christensen method with $m = 1.4n$.  Also, once again we see that the convergence rate for our method is nearly identical to that of the standard Fourier approximation.
\begin{table}[htbp]
\centering
\begin{tabular}{r| c c c|c c | c c}
\multirow{2}{*}{$n$} & \multicolumn{3}{c|}{error} & \multicolumn{2}{c|}{iterations} & \multicolumn{2}{c}{condition number}\\
 & (a) & (b) & (c) & (a) & (b) & (a) & (b)\\
\hline
16   & 3.3E-2  &  2.1E-5  &  2.1E-5   & 30   &  18  &  20.8 &  4.5  \\
32 & 9.9E-3  &  2.0E-6   & 2.0E-6   &  35  &  18  &  24.4 &  4.9  \\
64  & 7.0E-4 &  2.0E-7  &  1.9E-7   &  40 &  18  &  28.3 &  5.3 \\
128  & 3.9E-3 &  2.1E-8  &  2.0E-8   &  41  &  19  &  30.1 &  5.5  \\
256 & 1.7E-2  &  2.8E-9  &  2.1E-9  &  46 &  21 &  30.3 &  6.1 \\
\end{tabular}
\caption{Comparison of (a) the Casazza-Christensen method, (b) our new method (\ref{reconstruction}), and (c) the standard Fourier reconstruction method for Example \ref{example3}. }\label{table:smooth3}
\end{table}

\begin{figure}[htbp]
\centering
\subfloat[]{\label{3a}\includegraphics[scale=0.15]{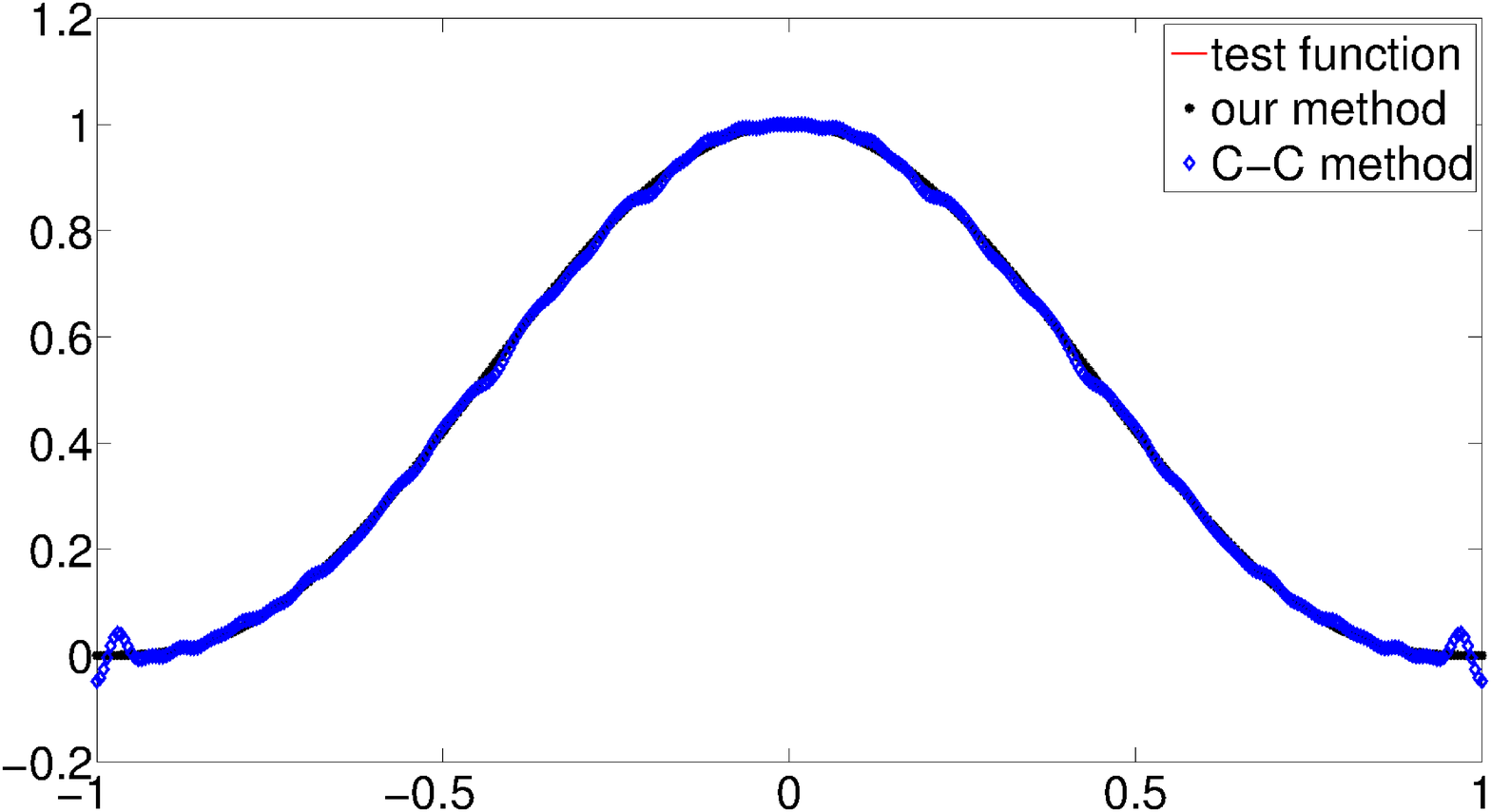}}\\
\subfloat[]{\label{3b}\includegraphics[scale=0.15]{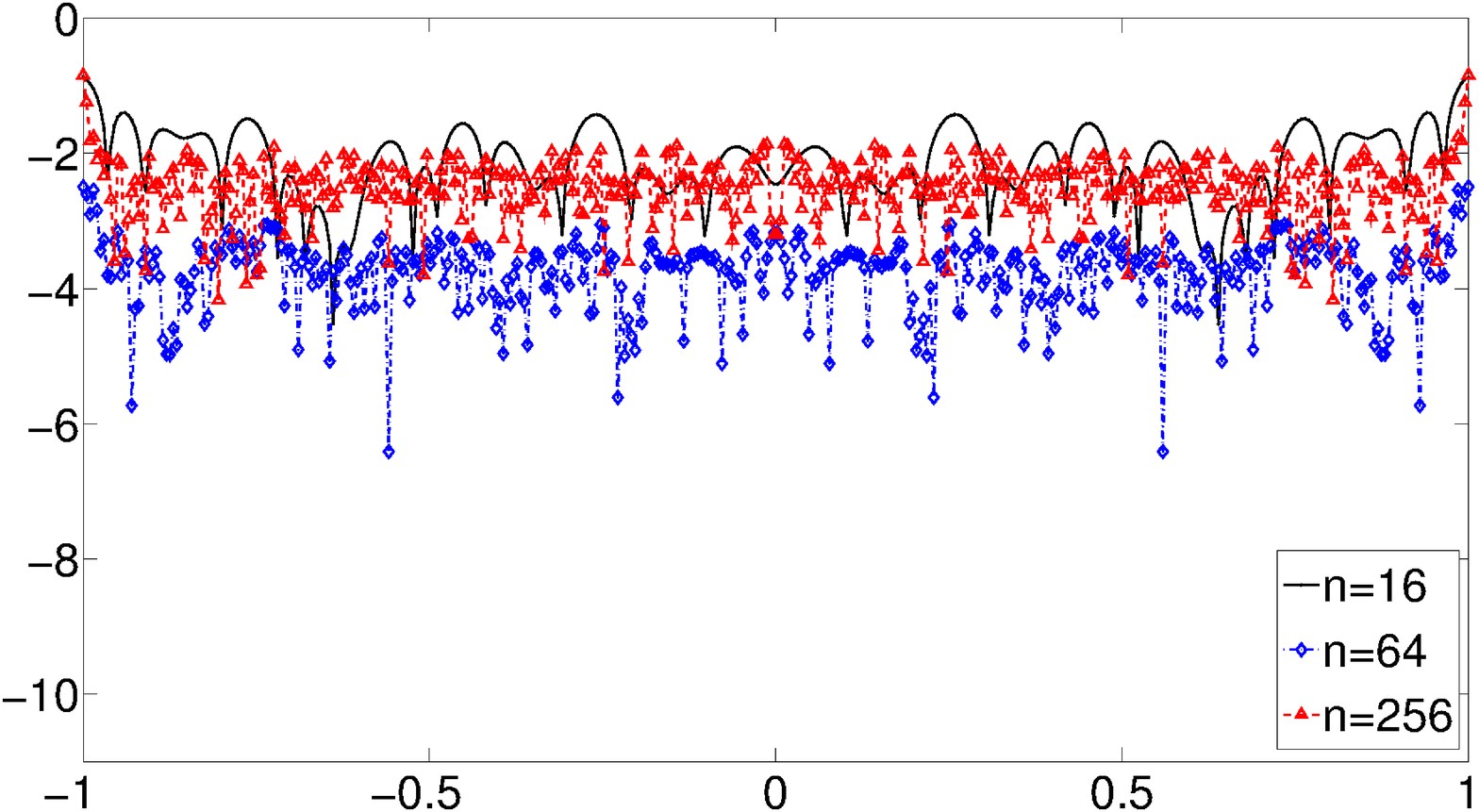}}
\subfloat[]{\label{3c}\includegraphics[scale=0.15]{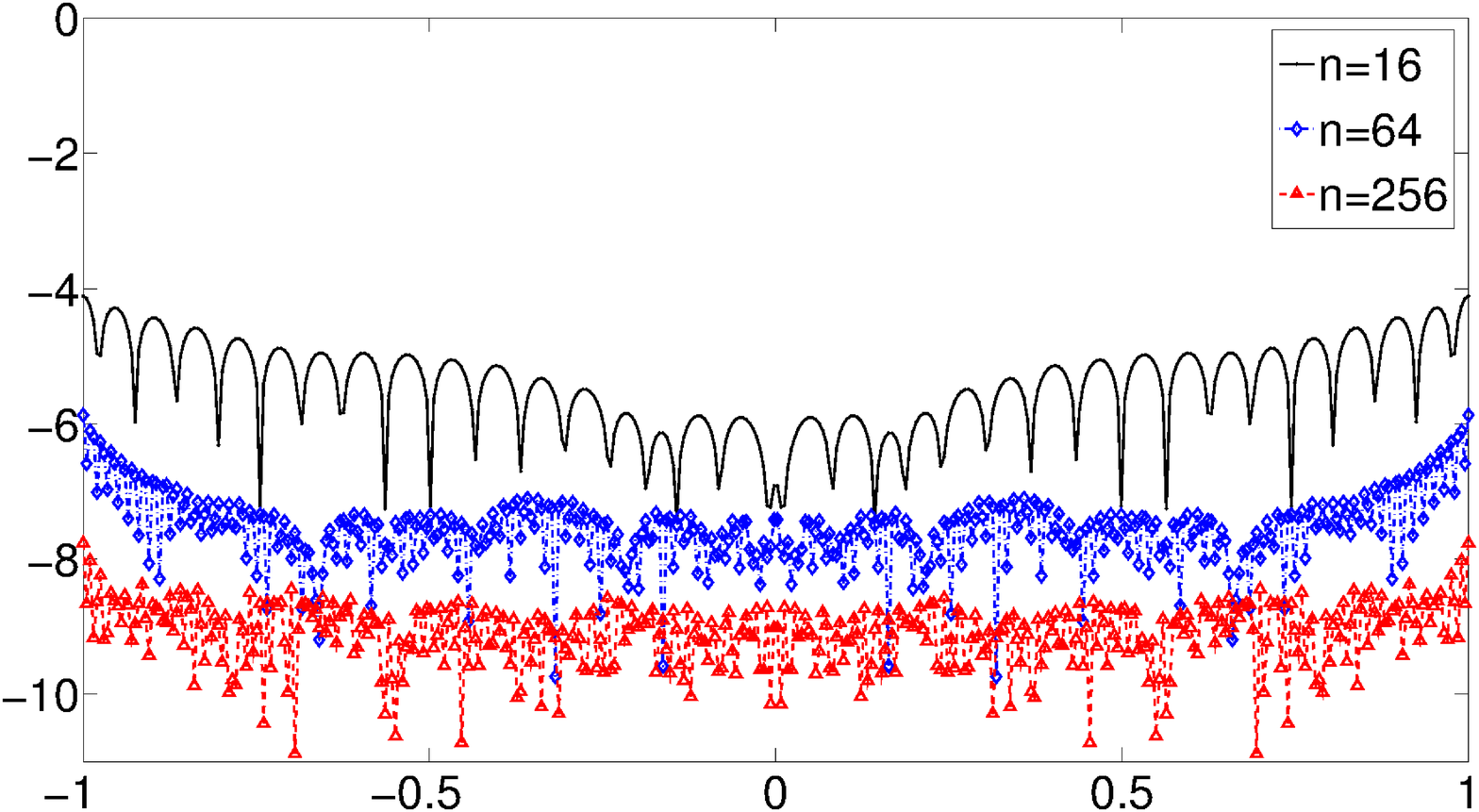}}
\caption{(a) Reconstruction of $f(x)=(1-x^2)^3$; (b) Point-wise $log_{10}$ error for the Casazza-Christensen method; (c) Point-wise $log_{10}$ error for (\ref{reconstruction}); }
\label{figure:smooth3}
\end{figure}

\begin{rem}\label{rem:remark5}
It is evident that our numerical frame approximation (\ref{reconstruction}) depends upon the convergence properties of the admissible frame.  In particular, if a Fourier basis is used, the reconstruction depends upon the smoothness and periodicity of the target function $f$.  
Hence when the target function is not smooth or not periodic, the method will suffer from the Gibbs phenomenon. One possible way to overcome this difficulty is to employ a post-processing technique, such as filtering or spectral reprojection, on the reconstruction.   In fact, it was shown in \cite{GelbHines} that it is possible to obtain exponential convergence when recovering piecewise smooth functions using spectral reprojection for frames.  On the other hand, we may consider the projection on some other well-localized frames such as polynomial frames instead of the Fourier basis used in approximation of the inverse frame operator. In other words, we should identify some well-localized frames that can fit in our setting and represent the unknown target function well without the Gibbs phenomenon. We shall leave these ideas to future investigations.
\end{rem}

\section{Concluding Remarks}
\label{section:conclusion}
In this investigation we constructed an approximation to the inverse frame operator.  We then used this approximation to develop a new reconstruction method when given a finite number of frame coefficients.  Our method is especially useful
when the original frame coefficients are not well localized, that is, the frame has localization rate no more than $1$.    It is important to point out that the number of samples required for our method is typically of the same order as the number of terms in the reconstruction.  The method  can also be used to improve the convergence rate in the case when the frame has localization rate greater than $1$. This is done through the introduction of {\em admissible} frames and the projection from the space spanned by the original frame elements onto the finite-dimensional subspace spanned by the admissible frame. Our numerical results demonstrate that our new method provides faster convergence with fewer iterations than the Casazza Christensen method, and moreover, the decay rate of the projected coefficients is the same as if the samples were originally given on the admissible frame.  Because of this it appears that in most cases a Riesz basis should be used as the admissible frame, since (1) it means that fewer samples $m$ are originally required and (2) it generates a more robust approximation of the inverse frame operator.  
However, when a sparse representation is necessary, a redundant frame may be more suitable.  

As discussed in Section \ref{section:reconstruction}, in the case of using the Fourier basis as the admissible frame, the reconstruction will yield the Gibbs phenomenon for piecewise smooth functions.  The spectral reprojection method \cite{GelbHines} may be used to post-process the reconstruction and recover exponential convergence.  On the other hand, it may prove to be more useful to use an admissible frame that makes different smoothness assumptions on the target function.

Finally, in this study  we considered only the noise-free case. When the sampling data is noisy, regularization techniques can be incorporated into our approach to obtain a robust and efficient approximation of both the inverse frame operator and the target function.   This idea, along with the others discussed in the preceding paragraphs, will be addressed  in future investigations.

\bibliographystyle{siam}
\bibliography{frame}

\end{document}